\documentclass[dvipdfmx,10pt]{amsart}
\usepackage{url}
\usepackage[dvipdfmx,colorlinks=true]{hyperref}
\numberwithin{equation}{section}
\theoremstyle{definition}
\newtheorem{thm}{Theorem}[section]
\newtheorem{prop}[thm]{Proposition}

\newtheorem{lem}[thm]{Lemma}
\newtheorem{rem}[thm]{Remark}

\newtheorem*{ack}{Acknowledgments}
\newtheorem*{mt}{Main Theorem}

\def\im{\mathop{\mathrm{Im}}\nolimits}

\title[outer Galois points and automorphisms of order 4]
{Quartic surfaces with an outer Galois point and $K3$ surfaces with an automorphism of order 4}
\author[K.~Miura]{Kei Miura}
\address{Department of Mathematics, National Institute of Technology, Ube College, 
2-14-1 Tokiwadai, Ube, Yamaguchi 755-8555, Japan}
\email{kmiura@ube-k.ac.jp}
\author[S.~Taki]{Shingo Taki}
\address{Department of Mathematics, Tokai University,
4-1-1, Kitakaname, Hiratsuka, Kanagawa, 259-1292, Japan}
\email{staki@tokai.ac.jp}
\urladdr{https://taki.sm.u-tokai.ac.jp}
\date{\today}
\dedicatory{}
\thanks{}
\begin{document}

\begin{abstract}
We prove that there exists a one-to-one correspondence between
smooth quartic surfaces with an outer Galois point
and $K3$ surfaces with a certain automorphism of order 4.
Furthermore, we characterize 
quartic surfaces with two or more outer Galois points 
as $K3$ surfaces.
\end{abstract}

\maketitle


\section{Introduction}\label{Introduction}

We will work over $\mathbb{C}$, the field of complex numbers, throughout this paper.
Let $V$ be a smooth hypersurface in $\mathbb{P}^{N}$
and $\mathbb{C}(V)$ the function field of $V$.
For a point $P \in \mathbb{P}^{N}$, we consider a 
projection $\pi_{P}: V \dashrightarrow H$ where 
$H$ is a hyperplane in $\mathbb{P}^{N}$ which does not contain $P$.
Note that the projection $\pi_{P}$ induces the extension 
$\mathbb{C}(V)/\pi_{P}^{\ast}\mathbb{C}(H)$ of function fields.
If the extension is Galois and $P \not \in V (\text{respectively } P \in V)$
then the point $P$ is called an \textit{outer} (respectively \textit{inner}) \textit{Galois point} for $V$.

Galois points were pioneered by H.~ Yoshihara in 1996, and have been studied by many mathematicians.
Here is a fundamental problem about Galois points:
\begin{itemize}
\item Find the number of (inner or outer) Galois points of a given hypersurface.
\item Characterize hypersurfaces with Galois points.
\end{itemize}

Yoshihara \cite{Yoshihara-GK3} determined equations of quartic surfaces with inner Galois points,
and found the number of inner Galois points.
Recently, we obtained a characterization of quartic surfaces with an inner Galois point 
in terms of $K3$ surfaces with an automorphism.

\begin{thm}[\cite{Miura-Taki}]
\begin{enumerate}
\item There exists a one-to-one correspondence between
smooth quartic surfaces with one inner Galois point
and 
$K3$ surfaces with a (purely) non-symplectic automorphism of order 3 and type $(4, 3)$.

\item The smooth quartic surface with the maximum number ($=8$) of inner Galois points
 is the singular $K3$ surface whose transcendental lattice is 
given by $\begin{pmatrix}8 & 4\\ 4 & 8\end{pmatrix}$.
\end{enumerate}
\end{thm}

Regarding outer Galois points, the case of quartic curves has been well 
studied by Miura and Yoshihara \cite{MY}.
However, there are few research results on outer Galois points in surfaces.
The following is the main theorem of this paper.

\begin{mt}
\begin{enumerate}
\item\label{mt1} There exists a one-to-one correspondence between
smooth quartic surfaces with one outer Galois point
and 
$K3$ surfaces with a purely non-symplectic automorphism of order 4 and type $(1, 0, 0, 3)$.

\item\label{mt2} The moduli space of smooth quartic surfaces with two outer Galois points is  
a subspace of codimension 1 in the moduli space of $K3$ surfaces with 
a non-purely non-symplectic automorphism of order 4 and type $(10, 4, 8)$.

\item\label{mt3} The smooth quartic surface with the maximum number ($=4$) of outer Galois points
 is the singular $K3$ surface whose transcendental lattice is 
given by $\begin{pmatrix}8 & 0 \\ 0 & 8\end{pmatrix}$.
\end{enumerate}
\end{mt}

\begin{ack}
This work was supported by JSPS KAKENHI Grant Number JP18K03230 and JP23K03036.
\end{ack}

\section{Preliminaries}\label{prelim}
In this section, we recall some basic results about 
smooth quartic surfaces with an outer Galois point, and
$K3$ surfaces with an automorphism of order 4.

\subsection{Quartic surfaces with an outer Galois point}
See \cite{Yoshihara--hyper} for details.
Let $P \in \mathbb{P}^{N}$ be a Galois point. 
Then projection $\pi_{P}: V \dashrightarrow H$ induces 
the Galois extension $\mathbb{C}(V)/\pi_{P}^{\ast}\mathbb{C}(H)$.
We denote by $G_{P}$ its Galois group.

\begin{prop}\label{equations}(\cite[Theorem 10, Corollary 6, Lemma 3]{Yoshihara--hyper}) 
Let $S$ be a smooth quartic surface in $\mathbb{P}^{3}$ with an outer Galois point.
\begin{itemize}
\item[(I)]
The homogeneous equation of $S$ is given by one of the followings (up to projective transformations): 
\begin{enumerate}
\item $X^{4}+F_{4}(Y, Z, W)=0$,
\item $X^{4}+Y^{4}+F_{4}(Z, W)=0$,
\item $X^{4}+Y^{4}+Z^{4}+W^{4}=0$, 
\end{enumerate}
where $F_{d}$ is a homogeneous polynomial of degree $d$ with distinct factors.
Moreover, the outer Galois points on each $S$ are
\begin{enumerate}
\item $[1:0:0:0]$,
\item $[1:0:0:0], [0:1:0:0]$,
\item $[1:0:0:0], [0:1:0:0], [0:0:1:0], [0:0:0:1]$, 
\end{enumerate}
respectively. 

\item[(II)] For outer Galois points $P_{1}=[1:0:0:0], P_{2}=[0:1:0:0], P_{3}=[0:0:1:0], P_{4}=[0:0:0:1]$,
each $G_{P_{k}}$ is the cyclic group of order 4.
Moreover if $\sigma_{k}$ is a generator of $G_{P_{k}}$ 
then $\sigma_{k}$ is given by one of the followings:
\begin{align*}
\sigma_{1} &: [X:Y:Z:W] \mapsto [\zeta_{4}X:Y:Z:W], \\
\sigma_{2} &: [X:Y:Z:W] \mapsto [X:\zeta_{4}Y:Z:W], \\
\sigma_{3} &: [X:Y:Z:W] \mapsto [X:Y:\zeta_{4}Z:W], \\
\sigma_{4} &: [X:Y:Z:W] \mapsto [X:Y:Z:\zeta_{4}W].
\end{align*}
\end{itemize}
where $\zeta_{4}$ is a primitive $4$-th root of unity. 
\end{prop}

\begin{rem}
Yoshihara gives results on a general hypersurface in $\mathbb{P}^{N}$ of degree $d$.
The above results are for cases of $d=4$ and $N=3$.
\end{rem}

\subsection{$K3$ surfaces with an automorphism of order 4}
For the details about foundations of $K3$ surfaces, see \cite{Huybrechts, Kondo-book}, and so on.
Let $S$ be a $K3$ surface, 
$\omega _{S}$ a nowhere vanishing holomorphic 2-form on $S$,
and $\sigma$ an automorphism of finite order $I$ on $S$.
\begin{enumerate}
\item $\sigma$ is called \textit{symplectic} if it satisfies $\sigma^{\ast} \omega _{S}=\omega _{S}$.
\item  $\sigma$ is called \textit{purely non-symplectic}  if it satisfies 
$\sigma^{\ast} \omega _{S}=\zeta_{I}\omega _{S}$
where $\zeta_{I}$ is a primitive $I$-th root of unity. 
\end{enumerate}

The following are well-known. See also \cite[\S 5]{Ni} and \cite[Theorem 0.1]{AST}.

\begin{lem}\label{fixedlocus}
We consider the fixed locus $S^{\sigma}=\{x \in S \mid \sigma(x)=x\}$ of $\sigma$.
\begin{enumerate}
\item If $\sigma$ is symplectic then $S^{\sigma}$ consists of isolated points.
\item If $\sigma$ is purely non-symplectic then $S^{\sigma}$ is either empty
 or the disjoint union of non-singular curves and isolated points.
 In particular, if $\sigma$ is of order 2 then $S^{\sigma}$ does not have isolated points.
\item If $\sigma$ is non-purely non-symplectic then $S^{\sigma}$ consists of isolated points.
\end{enumerate}
\end{lem}

Let $S$ be a $K3$ surface with a non-symplectic automorphism $\sigma$ of order 4
and $r$ the rank of the invariant lattice 
$L(\sigma)=\{x \in H^{2}(S, \mathbb{Z}) \mid \sigma^{\ast}(x)=x\}$.
We remark that the fixed locus $S^{\sigma}$ consists of 
$n$ isolated points, $k$ smooth rational curves, and at most one smooth curve of genus $g$.
Moreover, we denote by $2a$ the number of smooth curves fixed by $\sigma^{2}$ and
interchanged by $\sigma$.

See \cite{AS} for more details on non-symplectic automorphisms of order 4.
The following which holds from \cite[Theorem 0.1 and Proposition 2]{AS} is important in this paper.

\begin{prop}\label{ASkaraWakaru}
The followings hold:
\begin{enumerate}
\item Let $\sigma$ be a purely non-symplectic automorphism of order 4.
If the fixed locus $S^{\sigma}$ contains a smooth curve of genus $g>1$ then
$(r, k, a, g)=(1, 0, 0, 3)$, $(4, 0, 0, 2)$, $(2, 0, 1, 3)$, $(5, 0, 1, 2)$ or $(6, 0, 2, 2)$.
Moreover the number of isolated fixed points $n$ is given by
$n=2\sum_{C \subset S^{\sigma}}(1-g(C))+4$.

\item Let $\sigma$ be a non-purely non-symplectic automorphism of order 4,
hence $\sigma$ satisfies $\sigma^{\ast}\omega_{S}=-\omega$.
Then $S^{\sigma}$ contains $n \leq 8$ isolated points and
$(r, l, n)=(6,8,0), (7,7,2), (8,6,4), (9,5,6)$ or $(10, 4, 8)$.
Here $l$ is the rank of the eigenspace of $\sigma^{\ast}$ in $H^{2}(S, \mathbb{C})$
relative to the eigenvalue $-1$.
\end{enumerate}
\end{prop}

\section{The case of one Galois point}
In this section, we discuss Main Theorem (\ref{mt1}).
The proof is given by Proposition \ref{mt1-1} and Proposition \ref{mt1-2}.

\begin{prop}\label{mt1-1}
Let $S$ be a smooth quartic surface with an outer Galois point $P$.
Then the generator of $G_{P}$ is a purely non-symplectic automorphism of order 4 and type $(1, 0, 0, 3)$.
\end{prop}
\begin{proof}
We may assume that $S$ is given by $X^{4}+F_{4}(Y, Z, W)=0$ and 
the generator of $G_{P}$ is given by $\sigma_{1}$ by Proposition \ref{equations}.
Then we have
\begin{align*}
S^{\sigma_{1}}&=S \cap (\{X=0\}\amalg \{Y=Z=W=0\})\\
&=\{F_{4}(Y, Z, W)=0\}\\
&=C^{(3)},
\end{align*}
and 
\begin{align*}
S^{\sigma_{1}^{2}}&=S \cap (\{X=0\}\amalg \{Y=Z=W=0\})\\
&=\{F_{4}(Y, Z, W)=0\}\\
&=C^{(3)},
\end{align*}
where $C^{(g)}$ is a smooth curve of genus $g$.
This implies that $\sigma_{1}$ is of type $(1, 0, 0, 3)$ by Proposition \ref{ASkaraWakaru} (1).
\end{proof}

\begin{prop}\label{mt1-2}
Let $S$ be a $K3$ surface and $\sigma$ a purely non-symplectic automorphism of order 4 and type $(1, 0, 0, 3)$.
Then a pair $(S, \langle \sigma \rangle)$ 
gives a smooth quartic surface with an outer Galois point.
\end{prop}
\begin{proof}
Note that $\sigma^{2}$ is of order 2 and 
$S^{\sigma^{2}}$ consists of $C^{(3)}$ by Proposition \ref{ASkaraWakaru} (1)
and Lemma \ref{fixedlocus} (2).
Then there exists the rational map $\phi: S \dashrightarrow \mathbb{P}^{3}$ 
associated to the linear system $|C^{(3)}|$.
For $\phi$ to be an embedding, it is sufficient that
the intersection number of $C^{(3)}$ and any elliptic curve on $S$
is not 2 by \cite[Theorem 5.2 and Theorem 6.1]{SD}.

Let $E$ be an elliptic curve on $S$ and $m$ the intersection number of $C^{(3)}$ and $E$.
We apply the Hurwitz formula for a morphism $E\to E/\sigma^{2}$ of degree 2.
Since the ramification locus is the intersections of $C^{(3)}$ and $E$, 
we have
\[2\cdot g(E)-2=2(2\cdot g(E/\sigma^{2})-2)+m.\]
This implies that $m$ is 0 or 4.
Thus $\phi$ is an embedding.

Since $C^{(3)}$ is a fixed cuve of $\sigma$, it preserves $\phi$.
Thus $\sigma$ induces a projective transformation $\tilde{\sigma}$
which fixes the hyperplane $H$ such that $\phi^{-1}(H)=C^{(3)}$.
By replacing coordinates of $\mathbb{P}^{3}$ such that $H=\{X=0\}$, 
we may assume that 
$\tilde{\sigma}$ satisfies $\tilde{\sigma}([X:Y:Z:W])=[\zeta_{4}X:Y:Z:W]$.
Since $\im \phi$ is invariant for the action of $\tilde{\sigma}$,
we find that its equation is of the form in Proposition \ref{equations} (I)-(1).
\end{proof}

\begin{rem}
(1) The embedding $\phi$ is called a Galois embedding \cite{Yoshihara-embedding}. 

(2) We remark that the invariant lattice 
$L(\sigma^{2})=\{x \in H^{2}(S, \mathbb{Z}) \mid (\sigma^{2})^{\ast}(x)=x\}$
is isomorphic to $U(2)\oplus A_{1}^{\oplus 4}$ by \cite[Theorem 4.2.2]{Ni3}.
Let $\{e_{1}, e_{2}\}$ be a basis of $U(2)$ with
$\langle e_{1}, e_{1} \rangle =0, \langle e_{2}, e_{2} \rangle =0, \langle e_{1}, e_{2} \rangle =2$.
Then we may assume that  $e_{1}$ defines an elliptic fibration $\pi :S \to \mathbb{P}^{1}$.

But the rank of the invariant lattice $L(\sigma)$ is 1.
Thus, $\sigma$ replaces $\pi$ with another elliptic fibration, and 
it does not induce an automorphism of order 4 on an elliptic curve
which is a generic fiber of $\pi$.
\end{rem}

\section{The case of two Galois points}
In this section, we study smooth quartic surfaces with two outer Galois points.
Hence we treat Main Theorem (\ref{mt2}).

\begin{prop}\label{OGalois2ko}
Let $S$ be a smooth quartic surface with two outer Galois points $P_{1}$ and $P_{2}$.
Then there exists a non-purely non-symplectic automorphism $\sigma$  of order 4 and type $(10, 4, 8)$.
\end{prop}
\begin{proof}
We may assume that $S$ is given by $X^{4}+Y^{4}+F_{4}(Z, W)=0$ and 
each generator of $G_{P_{i}}$ is given by $\sigma_{i}$ by Proposition \ref{equations}.
Put $\sigma:=\sigma_{1} \circ \sigma_{2}$. Then we have
\begin{align*}
S^{\sigma}&=S \cap (\{X=Y=0\}\amalg \{Z=W=0\})\\
&=\{F_{4}(Z, W)=0\}\amalg \{X^{4}+Y^{4}=0\} \\
&=\{ \text{8 points} \}.
\end{align*}
Thus $\sigma$ is of type $(10, 4, 8)$ by Proposition \ref{ASkaraWakaru} (2).
\end{proof}

\begin{rem}\label{dimMnpns}\cite[Remark 1.3]{AS}
The dimension of the moduli space of $K3$ surfaces with 
a non-purely non-symplectic automorphism of order 4 is $l-2$.
\end{rem}

Let $\mathcal{M}$ be the moduli space of smooth quartic surfaces with two outer Galois points.

\begin{prop}
The space $\mathcal{M}$ is a subspace of codimension 1 
in the moduli space of $K3$ surfaces with 
a non-purely non-symplectic automorphism of order 4 and type $(10, 4, 8)$.
\end{prop}
\begin{proof}
Note that points of $\mathcal{M}$ represent 
isomorphism classes of smooth quartic surfaces with two outer Galois points.
If a smooth quartic surface has two outer Galois points then
we consider the following family in $\mathbb{P}^{3}$:
\[ \left\{ a_{1}X^{4}+a_{2}Y^{4}+\sum_{i=0}^{4}a_{i+3}Z^{i}W^{4-i}=0 \right\} \]
by Proposition \ref{equations}.
Since the subgroup
\[ \left\{ \begin{pmatrix} 
b_{1} & 0 & 0 & 0 \\
0 & b_{2} & 0 & 0 \\
0 & 0 & b_{3} & b_{4} \\
0 & 0 & b_{5} & b_{6} \\
  \end{pmatrix} 
\right\} \subset GL(4, \mathbb{C}) \]
commuting with $\sigma_{1}$ and $\sigma_{2}$ has the dimension 6, 
$\dim \mathcal{M}=7-6=1$.

Since the dimension of the moduli space of $K3$ surfaces with 
a non-purely non-symplectic automorphism of order 4 and type $(10, 4, 8)$
is $4-2=2$ by Remark \ref{dimMnpns}, the assertion holds.
\end{proof}

\section{The case of four (maximum) Galois points}
In this section, we treat Main Theorem (\ref{mt3}).
This is the case when a smooth quartic surface has the maximum number of outer Galois points.

\begin{prop}
Let $S$ be a smooth quartic surface with four Galois points.
\begin{enumerate}
\item $S$ is singular, that is, its Picard number is 20.
\item The Gram matrix of the transcendental lattice of $S$ is 
$\begin{pmatrix}8 & 0 \\ 0 & 8\end{pmatrix}$.
\end{enumerate}
\end{prop}
\begin{proof}
We may assume that $S$ is defined by the equation 
$X^{4}+Y^{4}+Z^{4}+W^{4}=0$ by Proposition \ref{equations}.
Thus the claim holds.
See also \cite[p.48]{Huybrechts}.
\end{proof}

\begin{prop}(\cite[Theorem 4]{SI})
Let $\mathcal{Q}$ be the set of $2\times 2$ positive definite even integral 
matrices:
\[ \mathcal{Q}=\left\{ \begin{pmatrix} 2a & b \\ b & 2c \end{pmatrix} \mid 
a,b,c \in \mathbb{Z}, a, c >0, b^{2}-4ac<0 \right\}.\]
There exists a bijective correspondence from the set of isomorphism classes of
singular $K3$ surfaces onto $\mathcal{Q}/SL_{2}(\mathbb{Z})$
which is given by a singular $K3$ surface maps to the Gram matrix of its transcendental lattice.
\end{prop}

The smooth quartic surface with four Galois points is 
characterized as the singular $K3$ surface whose 
transcendental lattice is $\begin{pmatrix}8 & 0 \\ 0 & 8\end{pmatrix}$.


\begin{thebibliography}{10}

\bibitem{AS}
M.~Artebani, A.~Sarti, 
Symmetries of order four on $K3$ surfaces, 
 J. Math. Soc. of Japan. Vo. \textbf{67}, No. 2 (2015), 1--31.

\bibitem{AST}
M.~Artebani, A.~Sarti, S.~Taki, 
$K3$ surfaces with non-symplectic automorphisms of prime order, 
Math. Z. \textbf{268} (2011), 507--533.

\bibitem{Huybrechts}
D.~Huybrechts,
\textit{Lectures on K3 surfaces},
Cambridge Studies in Advanced Mathematics, \textbf{158}. 
Cambridge University Press, Cambridge, 2016. 

\bibitem{Kondo-book}
S.~Kondo, 
\textit{K3 surfaces},
EMS Tracts in Mathematics, \textbf{32},
European Mathematical Society Publishing House, 2020.

\bibitem{Miura-Taki}
K.~Miura, S.~Taki, 
Quartic surfaces with a Galois point and Eisenstein $K3$ surfaces, 
preprint (arXiv:2308.11102).

\bibitem{MY}
K.~Miura, H.~Yoshihara,
Field theory for function fields of plane quartic curves,
J. Algebra \textbf{226} (2000), 283--294.

\bibitem{Ni}
V.V.~Nikulin, 
Finite automorphism groups of K\"ahlerian $K3$ surfaces, 
Trans. Moscow Math. Soc. \textbf{38} (1980), No 2, 71--135.

\bibitem{Ni3}
V.V.~Nikulin, 
Factor groups of groups of automorphisms of hyperbolic forms with
respect to subgroups generated by 2-reflections, 
J. Soviet Math. \textbf{22} (1983), 1401--1475.


\bibitem{SD}
D. ~Saint-Donat,
Projective models of $K3$ surfaces, 
Am. J. Math. \textbf{96}(4) (1974), 602--639.

\bibitem{SI}
T.~Shioda, H.~Inose, 
On singular $K3$ surfaces, 
Complex analysis and algebraic geometry,
119--136. Iwanami Shoten, Tokyo, 1977.

\bibitem{Yoshihara-GK3}
H.~Yoshihara, 
Galois points on quartic surfaces,
J. Math. Soc. Japan \textbf{53} (2001), 731--743.

\bibitem{Yoshihara--hyper}
H.~Yoshihara,
Galois points for smooth hypersurfaces,
J. Algebra \textbf{264} (2003), 520--534.

\bibitem{Yoshihara-embedding}
H.~Yoshihara, 
Galois embedding of algebraic variety and its application to abelian surface,
Rend. Sem. Mat. Univ. Padova \textbf{117} (2007), 69--85.


\end{thebibliography}
\end{document}